\documentclass[12pt]{amsart}
\usepackage{amsmath,amssymb,txfonts}
\usepackage{amssymb}
\usepackage{amsxtra}
\usepackage{amsmath}
\usepackage{txfonts}
\usepackage{enumerate}

\allowdisplaybreaks
\newtheorem{lemma}{Lemma}

\newtheorem{theorem}{Theorem}
\newtheorem{corollary}{Corollary}

\newtheorem{remark}{Remark}

\newtheorem{thma}{Theorem}

\def\bl{\begin{lemma}}
\def\bt{\begin{theorem}}
\def\el{\end{lemma}}
\def\et{\end{theorem}}
\def\bp{\begin{proof}}
\def\ep{\end{proof}}
\def\bc{\begin{corollary}}
\def\ec{\end{corollary}}
\def\bc{\begin{remark}}
\def\ec{\end{remark}}
\subjclass[2000]{Primary 47G10; Secondary30H20,47B38, 45P05.} 
\date{\today}
\begin{document}

\title[Essential norm]{Essential Norms and Weak Compactness of Integration Operators between Weighted Bergman Spaces}%{Norms, Essential Norms and Weak Compactness  of Integral Operators between Weighted Bergman Spaces}
\author{Santeri Miihkinen, Pekka J. Nieminen, and Wen Xu }
\address{Santeri Miihkinen:\ Department of Mathematics and Statistics, University of Helsinki, Box 68, 00014 Helsinki, Finland}
         \email{santeri.miihkinen@helsinki.fi}

\address{Pekka J. Nieminen:\ Department of Mathematics and Statistics, University of Helsinki, Box 68, 00014 Helsinki, Finland}
         \email{pjniemin@cc.helsinki.fi}

\address{Wen Xu:\ Department of Physics and Mathematics, University of Eastern Finland, P. O. Box 111, FI-80101, Joensuu, Finland}
          \email{wenx@uef.fi}
\thanks{}

\begin{abstract}
We consider Volterra-type integration operators $T_g$ %$$T_gf(z)=\int_0^z f(\zeta)g'(\zeta) d\zeta$$
between Bergman spaces induced by weights %$\omega$ in a large class $\widehat{D}$ of weights
satisfying a doubling property. We derive estimates for the operator norms, essential and weak essential norms of $T_g: A_\omega^p \to A_\omega^q$, $0<p\leq q<\infty$. In particular, the operator $T_g: A_\omega^1\to A_\omega^1$ is weakly compact if and only if it is compact.
\end{abstract}
 \maketitle
\section{Introduction}\label{s1}
Let $\mathbb D$ be the unit disk in the complex plane  and $\mathbb T$ be the boundary of $\mathbb D$. Let $H(\mathbb D)$ be the algebra of all analytic functions in $\mathbb D$. For $g\in H(\mathbb D)$, we consider the %generalized Volterra integration
generalized Volterra integration operator $T_g$ defined by
 $$T_g(f)(z)=\int_0^z f(\zeta)g'(\zeta)d\zeta, \quad z \in \mathbb{D}$$ for $f\in H(\mathbb D)$. The main purpose of the paper is to derive estimates for the operator norms and essential norms of $T_g: A_\omega^p\to A_\omega^q$, $0<p\leq q < \infty,$ as well as  weak essential norms of $T_g$ on $A_\omega^1$, where $A_\omega^p$ is the Bergman space induced by  $\omega$ in the class $\widehat{D}$ which consists of radial weights satisfying  the doubling property $\int_r^1\omega(s)ds\leq C\int_{\frac{1+r}{2}}^1 \omega(s)ds$ with $C = C(\omega) > 0$.  Essential norms of $T_g$ between classical weighted Bergman spaces have been estimated by R\"{a}tty\"{a} in \cite{R} for $1<p\leq q < \infty$. Later essential norms of $T_g$ on Hardy spaces,  BMOA and the Bloch space have been investigated  in \cite{LMN,LLX}.

Let $X$ and $Y$  be  complete metric spaces. For a bounded linear operator $T: ~ X \to ~ Y$,  the essential norm  (resp.  weak essential norm),  denoted by $\|T\|_{e,X\rightarrow Y}$  ( resp. $\|T\|_{w,X\to Y}$), is  the distance of $T$ (in the operator norm) from the closed ideal of compact operators (resp. weakly compact operators) $K: X\rightarrow Y$. Here an operator $K:X\to Y$ is weakly compact if $\overline{K(B)}$ is compact in the weak topology of Y, where $B$ is the unit ball of $X$. If either $X$ or $Y$ is reflexive, then every bounded operator $T: X\to Y$ is weakly compact. Since $A_\omega^1$ is nonreflexive, there are bounded operators on $A_\omega^1$ which are not weakly compact. If $A^1_\omega$ has so-called Schur property, i.e.\ weakly convergent sequences in $A^1_\omega$ are also norm convergent, then the class of weakly compact operators on $A^1_\omega$ coincides with the class of compact operators on $A^1_\omega$. We do not know if this is the case, therefore we also consider the weak compactness of the operator $T_g$ on $A^1_\omega$.

There are some previous results on the weak compactness of $T_g$. For example, it has been shown in \cite{LMN} that the compactness and weak compactness of the operator $T_g$ are equivalent on Hardy space $H^1$ and $BMOA.$ In the case of $BMOA$ a different proof of this fact was obtained in an independent work of Blasco et al. \cite{B} using different techniques.

The presence of large class of weights in our setting brings its own difficulties which were not present in the previous works concerning essential norms of operator $T_g$. For example, Littlewood-Paley type formula is usually used to get rid of the integral in the definition of $T_g.$ However, there is no such formula in general for $A^p_\omega, \, \omega \in \widehat{D}$ unless $p =2$, see \cite[Chapter 4]{PR}. In order to circumvent this problem we had to use different equivalent norms inherited from the theory of Hardy spaces, see \cite[Chapter 4]{PR}.

  For each radial weight $\omega$, its {\it associated weight} $\omega^*$ is defined by
$$\omega^*(z)=\int_{|z|}^1\omega(s)s\log\frac{s}{|z|}ds,\quad z\in \mathbb D\setminus\{0\}.$$ For $\alpha\geq 1$ and $\omega\in \widehat{D}$, the space $\mathcal C^\alpha(\omega^*)$ consists of $g\in H(\mathbb D)$ such that
 $$\|g\|_{\mathcal C^\alpha(\omega^*)}=|g(0)|+ \|g\|_{*,\alpha,\omega} <\infty,$$ where
$$\|g\|_{*,\alpha,\omega} = \sup_{I \subset \mathbb{T}} \sqrt{\frac{\int_{S(I)}|g'(z)|^2\omega^*(z)dA(z)}{(\omega(S(I)))^\alpha}}$$ is a seminorm on $C^\alpha(\omega^*)$, $S(I)=\{re^{i\theta}\in \mathbb D: e^{i\theta}\in I, 1-|I|\leq r<1\}$ is the Carleson square associated with $I\subseteq \mathbb T$, $|E|$ is the Lebesgue measure of $E\subseteq \mathbb T$  and $\omega(S(I)) = \int_{S(I)}\omega(z)dA(z)$. We associate each $a\in \mathbb D\setminus\{0\}$ with the interval $I_a=\left\lbrace e^{i\theta}:|\arg(ae^{-i\theta})|\leq \frac{1-|a|}{2} \right\rbrace$, and denote $S(a)=S(I_a)$. The space $C^\alpha_0(\omega^*)$ consists of $g\in H(\mathbb D)$ such that $$ \limsup_{|I| \to 0}\frac{\int_{S(I)}|g'(z)|^2\omega^*(z)dA(z)}{(\omega(S(I)))^\alpha}=0.$$

Throughout the paper the notation $A\lesssim B$ indicates that there is a constant $c$ independent of said or implied variables or functions such that $A\leq cB$. If $A\lesssim B$ and $B\lesssim A$, we write $A\simeq B$ and say that $A$ and $B$ are equivalent quantities.

The next result is a generalization of a part of Theorem 4.1 in \cite{PR} for the weights in the class $\widehat{D}$.
\begin{thma}\label{Tbound}
Let $0<p\leq q<\infty$, $\alpha = 2(\frac{1}{p}-\frac{1}{q})+1$, $\frac{1}{p}-\frac{1}{q}<1$, $\omega\in \widehat{D}$ and $g\in H(\mathbb D)$.
 Then $T_g:A_\omega^p \to A_\omega^q$ is bounded if and only if $g\in \mathcal C^\alpha(\omega^*)$.
\end{thma}

Below are our main results. The first result is a quantitative extension of Theorem \ref{Tbound}.

\begin{theorem}\label{norm}
Let $0<p\leq q <\infty$, $\omega\in \widehat{D}$, $\alpha = 2(\frac{1}{p}-\frac{1}{q})+1$ and $g\in C^\alpha(\omega^*)$. Then there exists $\eta=\eta(\omega)>1$ large enough such that the following quantities are comparable:
\begin{align*}
& \|T_g\|_{A_\omega^p\to A_\omega^q};\\
& \|g \|_{*,\alpha,\omega}=\sup_{I\subseteq \mathbb T}\left(\frac{\int_{S(I)}|g'(z)|^2\omega^*(z)dA(z)}{\omega(S(I))^\alpha}\right)^{1/2};\\
& B=\sup_{a\in \mathbb D}\int_{\mathbb D}\left(\frac{1}{\omega(S(a))}\left(\frac{1-|a|}{|1-\bar{a}z|}\right)^\eta\right)^{\alpha}|g'(z)|^2\omega^*(z)dA(z);\\
& C=\sup_{z \in \mathbb D}|g'(z)|(1-|z|)\omega^*(z)^{\frac{1}{q}-\frac{1}{p}}, p < q.
\end{align*}
Constants of comparison are independent of $g$.
\end{theorem}
\begin{theorem}\label{essnorm}
Let $0< p\leq q < \infty$, $\alpha = 2(\frac{1}{p}-\frac{1}{q})+1$, $\frac{1}{p}-\frac{1}{q}<1$, $\omega\in \widehat{D}$ and  $g\in \mathcal{C}^\alpha(\omega^*)$. Then there exists $\eta=\eta(\omega)>1$ large enough such that the following quantities are comparable: %Then the following quantities are comparable:
\begin{align*}
& \|T_g\|_{e,A_\omega^p\rightarrow A_\omega^q};\\
& A=\hbox{dist} (g,C_0^\alpha);\\
& B=\limsup_{|I|\to0}\left(\frac{\int_{S(I)}|g'(z)|^2\omega^*(z)dA(z)}{\omega(S(I))^\alpha}\right)^{1/2};\\
& C=\limsup_{|a|\to 1^-}\int_{\mathbb D}\left(\frac{1}{\omega(S(a))}\left(\frac{1-|a|}{|1-\bar{a}z|}\right)^\eta\right)^\alpha|g'(z)|^2\omega^*(z)dA(z);\\
& D=\limsup_{|z|\to 1^-}|g'(z)|(1-|z|)\omega^*(z)^{\frac{1}{q}-\frac{1}{p}}, p < q.\end{align*}
\end{theorem}

\begin{theorem}\label{theorem2}
Let $\omega\in \widehat{D}$ and  $g\in \mathcal{C}^1(\omega^*)$. Then
\begin{eqnarray*}
&&\|T_g\|_{w,A_\omega^1\rightarrow A_\omega^1}\simeq \hbox{dist} (g,C_0^1(\omega^*))\\
&&\simeq \limsup_{|I|\to 0}\left(\frac{\int_{S(I)}|g'(z)|^2\omega^*(z)dA(z)}{\omega(S(I))}\right)^{1/2}
\simeq \|T_g\|_{e,A_\omega^1\rightarrow A_\omega^1}.
\end{eqnarray*}
In particular, the operator $T_g$ is weakly compact on $A^1_\omega$ if and only if it is compact.
\end{theorem}

The paper is organized as follows. In section 2, we give some preliminary results. In section 3, the proofs of norm estimates  are presented. In section 4, we investigate essential norms between two weighted Bergman spaces and weak compactness on $A_\omega^1$.

\section{Preliminaries}
An integrable function  $\omega: \mathbb D\to (0,\infty)$ is called a \textit{weight  function} or simply a \textit{weight}. For $0<p<\infty$ and a weight $\omega$, the \textit{weighted Bergman space} $A_\omega^p$  stands for the space of all  functions $f\in H(\mathbb D)$ satisfying
$$\|f\|_{A_\omega^p}^p=\int_{\mathbb D}|f(z)|^p\omega(z)dA(z)<\infty,$$
where $dA(z)=\frac{1}{\pi}dxdy$ is the normalized Lebesgue area measure on $\mathbb D$. For $\omega(z)=(1-|z|^2)^\alpha$, $-1<\alpha<\infty$, $A_\omega^p$ is the classical weighted Bergman space. If $1\leq p<\infty$, then $\|\cdot\|_{A_\omega^p}$ is a norm which makes $A_\omega^p$ a Banach space.
But if $0<p<1$, then it is instead $\|\cdot\|_{A_\omega^p}^p$ which is subadditive and used to induce the complete translation invariant metric. The operator norm is defined as usual $$\|T_g\|_{A_\omega^p\rightarrow A_\omega^q}=\sup_{\|f\|_{A_\omega^p}\leq 1}\|T_gf\|_{A_\omega^q},$$ although in the case $0 < q < 1$ the quantity $\|\cdot\|_{A_\omega^p\rightarrow A_\omega^q}$ is a quasi-norm, but we make no distinction between that and the operator norm.
%Even though, we still define integral operator norm as usual, for $0<p,q<\infty$, integral operator norm is defined as usual $$\|T_g\|_{A_\omega^p\rightarrow A_\omega^q}=\sup_{\|f\|_{A_\omega^p}\leq 1}\|T_gf\|_{A_\omega^q}.$$

A weight $\omega$ is radial if $\omega(z)=\omega(|z|)$ for all $z\in \mathbb D$. Let $\widehat{D}$ be the class of radial weights such that $\widehat{\omega}(r)=\int_r^1\omega(s)ds$ satisfies  the doubling property, that is, there exists $C=C(\omega)$ such that  $$\widehat{\omega}(r)\leq C\widehat{\omega}\left(\frac{1+r}{2}\right),\quad \mbox{for}\  \forall\  0\leq r<1.$$
A radial weight $\omega$ is called \textsl{regular} if $\omega$ is continuous and satisfies $$\frac{\widehat{\omega}(r)}{\omega(r)}\simeq 1-r, \quad \mbox{for } 0\leq r<1.$$ The weight $\omega^*$ is regular if $\omega\in \widehat{D}$. The class of regular weights is denoted by $\mathcal{R}$. Also, a radial weight $\omega$ is in the class of rapidly increasing weights $\mathcal{I}$ if it is continuous and satisfies $$\lim_{r \to 1^-} \frac{\widehat{\omega}(r)}{(1-r)\omega(r)} = \infty.$$ See \cite{PR} for more information on classes $\mathcal{I}$ and $\mathcal{R}$.

Recall that non-tangential regions and the tents are defined by
$$\Gamma(u)=\left\{z\in \mathbb D: |\theta-\arg z|<\frac{1}{2}\left(1-\frac{|z|}{r}\right)\right\},
\quad u=re^{i\theta}\in \overline{\mathbb D}\setminus\{0\},$$$$T(z)=\{u\in \mathbb D:z\in \Gamma(u)\},\quad z\in \mathbb D.$$
A simple computation shows that $\omega(S(z))\simeq\omega(T(z))\simeq\omega^*(z)$, as $|z|\to 1^-$, provided $\omega\in \widehat{D}$.
The maximal function related to the measure $\omega(\cdot)dA$ is defined by
$$M_\omega(\psi)(z)=\sup_{I: z\in S(I)}\frac{1}{\omega(S(I))}\int_{S(I)}|\psi(\xi)|\omega(\xi)dA(\xi),\quad z\in \mathbb D,$$
 where $\psi\in L_\omega^1$. For more information on $A_\omega^p$, see \cite{PR2,PR1,PR}.

Recall that for a given Banach space (or a complete metric space) $X$ of analytic functions on $\mathbb{D}$, a positive Borel measure $\mu$ on $\mathbb{D}$ is called a q-Carleson measure for $X$ if the identity operator $I : X \to L^q(\mu)$ is bounded. Pel\'{a}ez and R\"{a}tty\"{a} \cite{PR} investigated the $q$-Carleson measure for $A_\omega^p$, as well as the boundedness  and compactness of the integral operator $T_g$, where $\omega \in \mathcal I \cup ~ \mathcal R$. The classes $\mathcal I$ and $\mathcal R$ are contained in $\widehat{D}$. In fact $\widehat{D}$ preserves almost all the properties of $\mathcal I \cup \mathcal R$ and so those statements concerning the Carleson measures and  the integral operators are also true on $A_\omega^p$, $\omega\in \widehat{D}$. For the reader's convenience, we list some results here and skip proofs. The next lemma is essentially Theorem 2.1 and Corollary 2.2 in \cite{PR}.

\begin{lemma}\label{eCarle}
Let $0<p \leq q <\infty$ and $\omega\in \widehat{D}$, and let $\mu$ be a positive Borel measure on $\mathbb D$. Then $\mu$ is a $q$-Carleson measure  for $A_\omega^p$ if and only if \begin{equation}\label{Carlecon}G\triangleq\sup_{I\subseteq \mathbb T}\frac{\mu(S(I))}{(\omega(S(I)))^{\frac{q}{p}}}<\infty.\end{equation} Moreover, if $\mu$ is a $q$-Carleson measure for $A_\omega^p$, then for all $f\in A_\omega^p$
\begin{equation}\label{eCarnorm}
\|f\|_{A_\mu^q}^q\lesssim G \|f\|_{A_\omega^p}^q.
\end{equation} Furthermore, if  $\alpha\in (0,\infty)$ such that $p\alpha>1$, then $ [M_\omega((\cdot)^{\frac{1}{\alpha}})]^\alpha: L_\omega^p\rightarrow L_\mu^q$ is bounded if and only if $\mu$ satisfies (\ref{Carlecon}) and $\|[M_\omega((\cdot)^{\frac{1}{\alpha}})]^\alpha\|_{L_\omega^p\rightarrow L_\mu^q}^q\simeq G$.

\end{lemma}

\begin{remark} The operator $\psi \mapsto M_\omega(\psi)$ is sublinear, but its norm is defined like in the case of a linear operator.
\end{remark}
See \cite[Theorem 4.2]{PR} for the next lemma.
\begin{lemma}
Let $0<p<\infty$, $n\in \mathbb N$ and $f\in H(\mathbb D)$, and let $\omega$ be a radial weight. Then
\begin{equation}\label{equiv1}
\|f\|_{A_\omega^p}^p=p^2\int_{\mathbb D}|f(z)|^{p-2}|f'(z)|^2\omega^*(z)dA(z)+\omega(\mathbb D)|f(0)|^p,
\end{equation}
and \begin{equation}\label{equiv2}
\|f\|_{A_\omega^p}^p\simeq\int_{\mathbb D}\left(\int_{\Gamma(u)}|f^{(n)}(z)|^2\left(1-\left|\frac{z}{u}\right|\right)^{2n-2}dA(z)\right)^{\frac{p}{2}}\omega(u)dA(u)+\sum_{j=0}^{n-1}|f^{(j)}(0)|^p,
\end{equation} where the constants of comparison depend only on $p,n$ and $\omega$. In particular,
\begin{equation}\label{equiv3}
\|f\|_{A_\omega^2}^2=4\|f'\|_{A_{\omega^*}^2}^2+\omega(\mathbb D)|f(0)|^2.
\end{equation}
\end{lemma}

Recall that the non-tangential maximal function of $f$ in the unit disk is defined by
$N(f)(u)=\sup_{z\in \Gamma(u)}|f(z)|$, $u\in \mathbb D\setminus\{0\}$. The following equivalent norm will be used in our proof also, see \cite[Lemma 4.4]{PR}.
\begin{lemma}\label{N(f)}
Let $0<p<\infty$ and let $\omega$ be a radial weight. Then $$\|N(f)\|_{A_\omega^p}\simeq\|f\|_{A_\omega^p},\quad \rm{for\ all}\ f\in A_\omega^p.$$
\end{lemma}

Proposition 4.7 in \cite{PR} also holds for weights in the class $\widehat{D}$ and it states that $f \in C^\alpha(\omega^*), \, \alpha > 1$ if and only if $$M_\infty(f',r) \lesssim \frac{(\omega^*(r))^{\frac{\alpha-1}{2}}}{1-r}, \quad 0 \le r < 1$$ and $f \in C_0^\alpha(\omega^*)$ if and only if $$M_\infty(f',r) = \textup{o}\left(\frac{(\omega^*(r))^{\frac{\alpha-1}{2}}}{1-r}\right), \quad r \to 1^-.$$
Furthermore, the proof of Proposition 4.7 in \cite{PR} implies that
\begin{lemma}\label{littleg}
Let $0<\alpha<\infty$, $\omega\in \widehat{D}$ and $g\in \mathcal{ C}^{2\alpha+1}(\omega^*)$. Then
\begin{equation}
\limsup_{|z|\to 1^-}|g'(z)|(1-|z|)\omega^*(z)^{-\alpha}=\limsup_{|a|\to 1^-}\left(\frac{\int_{S(a)}|g'(z)|^2\omega^*(z)dA(z)}{\omega(S(a))^{2\alpha+1}}\right)^{\frac{1}{2}}.\end{equation}
\end{lemma}
%The next result is a generalization of a part of Theorem 4.1 in \cite{PR} for the weights in the class $\widehat{D}$.
%\begin{theorem}\label{Tbound}
%Let $0<p\leq q<\infty$, $\frac{1}{p}-\frac{1}{q}<1$, $\omega\in \widehat{D}$ and $g\in H(\mathbb D)$.
% Then $T_g:A_\omega^p \to A_\omega^q$ is bounded if and only if $g\in \mathcal C^{2(\frac{1}{p}-\frac{1}{q})+1}(\omega^*)$.
%\end{theorem}

In the next lemma, we classify spaces $C^\alpha(\omega^*)$ and $C_0^\alpha(\omega^*)$ according to how fast the quantity $$\frac{(\omega^*(r))^{\frac{\alpha-1}{2}}}{1-r}$$ grows as $r \to 1^-$. The proof is straightforward and we omit it.

\begin{lemma}
\label{liminf_lemma}
Let $\omega \in \widehat{D}, \, \alpha > 1$ and $$F_{\alpha,\omega}(r) = \frac{(\omega^*(r))^{\frac{\alpha-1}{2}}}{1-r}, \quad r \in ]0,1[.$$ Define $$\beta=\liminf_{r \to 1^-}F_{\alpha,\omega}(r).$$ Then

\begin{enumerate}[(i)]
  \item \textrm{If $\beta = 0$, then $C_0^\alpha(\omega^*)=C^\alpha(\omega^*)=\{f \in H(\mathbb{D})|\, \textrm{$f$ is a constant function}\};$}
  \item \textrm{If $\beta \in ]0,\infty[$, then $C_0^\alpha(\omega^*)=\{f \in H(\mathbb{D})|\, \textrm{$f$ is a constant function}\}$ and $C^\alpha(\omega^*)=\{f \in H(\mathbb D)|\, f' \in H^\infty(\mathbb D)\};$}
  \item \textrm{If $\beta = \infty$, then $\{f \in H(\mathbb D)|\, f' \in H^\infty(\mathbb D)\} \subsetneq C_0^\alpha(\omega^*) \subset C^\alpha(\omega^*).$}
\end{enumerate}
\end{lemma}

%Define  \begin{equation}\label{fap}f_{a,p}(z)=\frac{(1-|a|)^{\frac{\gamma+1}{p}}}{(1-\bar{a}z)^{\frac{\gamma+1}{p}}\omega(S(a))^{\frac{1}{p}}},\end{equation} where $\gamma=\beta(\omega)>0$ is the constant in Lemma 1.1 \cite{PR}. A simple computation shows that $\sup_{a\in \mathbb D}\|f_{a,p}\|_{A_\omega^p}\lesssim 1$, and $f_a(z)\to 0$ uniformly on compact subsets of $\mathbb D$ as $|a|\to 1$. By the proof of Lemma 5.3 in \cite{PR}, for $\omega\in \widehat{D}$, we have $$\|g\|_{C^\alpha(\omega^*)}\simeq |g(0)|+\sup_{a\in \mathbb D}\|T_gf_{a,\frac{2}{\alpha}}\|_{A_\omega^2}\simeq |g(0)|+\sup_{a\in \mathbb D}\left(\int_{\mathbb D}|f_{a,\frac{2}{\alpha}}|^2|g'(z)|^2\omega^*(z)dA(z)\right)^{\frac{1}{2}},$$
%and \begin{equation}\label{limsupequiv}\limsup_{|a|\to 1}\frac{\int_{S(a)}|g'(z)|^2\omega^*(z)dA(z)}{\omega(S(a))^\alpha}\simeq \limsup_{|a|\to 1}\|T_gf_{a,\frac{2}{\alpha}}\|_{A_\omega^2}^2.\end{equation}

A function-theoretic quantity to estimate the distance of a general $C^\alpha(\omega^*)$-function from $C_0^\alpha(\omega^*)$ is given by
\begin{lemma}\label{distance} Let $\omega\in \widehat{D}$ and $\alpha\geq 1$.
For  $g\in C^\alpha(\omega^*)$,
$$\hbox{dist}(g, C_0^\alpha(\omega^*))\simeq \limsup_{|I|\to 0}\left(\frac{\int_{S(I)}|g'(z)|^2\omega^*(z)dA(z)}{\omega(S(I))^\alpha}\right)^{\frac{1}{2}}.$$
\end{lemma}

\begin{proof}
The lower estimate is trivial from the definitions of $C^\alpha(\omega^*)$ and
$C_0^\alpha(\omega^*)$.

For the upper estimate we consider three cases. Let $\beta$ be the number defined in Lemma \ref{liminf_lemma}.

{\it Case 1}- Assume $\alpha > 1$ and $\beta=0$.

It follows immediately from the case (i) of Lemma \ref{liminf_lemma} that $$\hbox{dist}(g, C_0^\alpha(\omega^*))\simeq \limsup_{|I|\to 0}\left(\frac{\int_{S(I)}|g'(z)|^2\omega^*(z)dA(z)}{\omega(S(I))^\alpha}\right)^{\frac{1}{2}}.$$

{\it Case 2}- Assume $\alpha > 1$ and $\beta \in ]0,\infty[$.

Define $$G_{\omega,g} \colon ]0,1] \to \mathbb{R_+}, \, G_{\omega,g}(t) = \sup_{|I| = t}\left(\frac{\int_{S(I)}|g'(z)|^2 \omega^*(z)dA(z)}{\omega(S(I))^\alpha}\right)^{1/2}$$ and $G = G_{\omega,\textup{id}}$. Now $\hbox{dist}(g, C_0^\alpha(\omega^*))=\sup_{t \in ]0,1]}G_{\omega,g}(t)$, since $$C_0^\alpha(\omega^*) = \{f \in H(\mathbb{D})|\, \textup{$f$ is a constant function}\}$$ by the case (ii) of Lemma \ref{liminf_lemma}. It is enough to show that $$\sup_{t \in ]0,1]}G_{\omega,g}(t) \lesssim \limsup_{t \to 0+}G_{\omega,g}(t),$$ since
the direction $$\limsup_{t \to 0+}G_{\omega,g}(t) \le \sup_{t \in ]0,1]}G_{\omega,g}(t)$$ is evident.

It holds that $\limsup_{t \to 0+}G(t) \in ]0,\infty[$, since  $\textup{id} \in C^\alpha(\omega^*) \setminus C_0^\alpha(\omega^*)$ by the case (ii) of Lemma \ref{liminf_lemma}. Now
\begin{equation}
\label{eq: sup_est}
\sup_{t \in ]0,1]}G(t) \simeq \limsup_{t \to 0+}G(t).
\end{equation}
Since $g' \in H^\infty(\mathbb{D})$, we can assume by rotation invariance that there exist the non-tangential limit $g'(1)=\lim_{\substack{z \to 1 \\ z \in N}}g'(z)$ s.t.\ $|g'(1)| > \frac{1}{2}\|g'\|_{H^\infty(\mathbb{D})}$, where $N \subset \mathbb{D}$ is any non-tangential set with vertex at $z = 1$. Also, there exist $r_0 \in [0,1[$, a Carleson window $S_0 = S(r_0)$ and a non-tangential set $T \subset S_0$ with vertex at $z=1$  s.t.\ $|g'(z)| \ge \frac{1}{2}\|g'\|_{H^\infty(\mathbb{D})}$ for all $z \in T$ and $\omega^*(T) \simeq \omega^*(S_0).$ Let $S=S(I)$ be any Carleson window s.t.\ $|I|\le 1-r_0$. Choose a Carleson window $S'=S'(I') \subset S_0$ with $|I'|=|I|$ and a non-tangential set $T' \subset S' \cap T$ with vertex at $z =1$ s.t.\  $\omega^*(T') \simeq \omega^*(S')$. Now we can estimate
\begin{eqnarray*}
\sup_{t \le 1-r_0}G_{\omega,g}(t) &\ge&
\left(\frac{\int_{S'}|g'(z)|^2\omega^*(z)dA(z)}{\omega(S')^\alpha}\right)^{\frac{1}{2}}
\\
&\ge&
\left(\frac{\int_{T'}|g'(z)|^2\omega^*(z)dA(z)}{\omega(S')^\alpha}\right)^{\frac{1}{2}} \gtrsim \| g' \|_{H^\infty(\mathbb{D})} \left(\frac{\omega^*(T')}{\omega(S')^\alpha}\right)^{\frac{1}{2}}
\\
&\simeq& \| g' \|_{H^\infty(\mathbb{D})} \left(\frac{\omega^*(S')}{\omega(S')^\alpha}\right)^{\frac{1}{2}} = \| g' \|_{H^\infty(\mathbb{D})} \left(\frac{\omega^*(S)}{\omega(S)^\alpha}\right)^{\frac{1}{2}}.
\end{eqnarray*}
Hence $$\sup_{t \le 1-r_0}G_{\omega,g}(t) \gtrsim \| g' \|_{H^\infty(\mathbb{D})} \sup_{t \le 1-r_0}G(t)$$ and letting $r_0 \to 1^-$ we get
\begin{equation}
\label{eq: limsup_est}
\limsup_{t \to 0+}G_{\omega,g}(t) \gtrsim \| g' \|_{H^\infty(\mathbb{D})} \limsup_{t \to 0+}G(t).
\end{equation}
Now by \eqref{eq: sup_est} and \eqref{eq: limsup_est} we get
\begin{eqnarray*}
\sup_{t \in ]0,1]}G_{\omega,g}(t) \le \|g'\|_{H^\infty(\mathbb{D})} \sup_{t \in ]0,1]}G(t) \simeq  \|g'\|_{H^\infty(\mathbb{D})} \limsup_{t \to 0+}G(t) \lesssim \limsup_{t \to 0+}G_{\omega,g}(t).
\end{eqnarray*}
Thus we have established the upper estimate in the case $\beta \in ]0,\infty[$.

{\it Case 3}- Assume $\alpha = 1$ or $\beta=\infty$.

Now it holds that $$\{f \in H(\mathbb{D})|\, f' \in H^\infty(\mathbb{D})\} \subset C_0^\alpha(\omega^*).$$ Set $g_r(z)=g(rz)$ for $0<r<1$. Then $g_r\in C_0^\alpha(\omega^*)$.  Fix $0 < \delta < 1$. Now
\begin{eqnarray*}
\textup{dist}(g, C_0^\alpha(\omega^*))^2 &\le& \limsup_{r \to 1^-}\| g - g_r\|_{C^\alpha(\omega^*)}^2
\\
&\le& \limsup_{r \to 1^-}\left(\sup_{|I| \ge \delta}\frac{1}{\omega(S(I))^\alpha}\int_{S(I)}|g'(z)-r g'(rz)|^2\omega^*(z)dA(z)\right.
\\
&+& \left.\sup_{|I| < \delta}\frac{1}{\omega(S(I))^\alpha}\int_{S(I)}|g'(z)-r g'(rz)|^2\omega^*(z)dA(z)\right)
\\
&=& \limsup_{r \to 1^-}\left(\sup_{|I| \ge \delta}\frac{1}{\omega(S(I))^\alpha}\int_{S(I)}|g'(z)-r g'(rz)|^2\omega^*(z)dA(z)\right)
\\
&+& \limsup_{r \to 1^-}\left(\sup_{|I| < \delta}\frac{1}{\omega(S(I))^\alpha}\int_{S(I)}|g'(z)-r g'(rz)|^2\omega^*(z)dA(z)\right),
\end{eqnarray*}
where $$\sup_{|I| \ge \delta}\frac{1}{\omega(S(I))^\alpha}\int_{S(I)}|g'(z)-r g'(rz)|^2\omega^*(z)dA(z) \lesssim \|g'-(g_r)'\|_{A^2_{\omega^*}}^2 \to 0, \quad r \to 1^-.$$ Thus we have
\begin{eqnarray}
\label{eq: distest1}
\textup{dist}(g, C_0^\alpha(\omega^*))^2 &\lesssim& \limsup_{r \to 1^-}\left(\sup_{|I| < \delta}\frac{1}{\omega(S(I))^\alpha}\int_{S(I)}|g'(z)-r g'(rz)|^2\omega^*(z)dA(z)\right)
\nonumber \\
&\lesssim& \sup_{r > 1 - \delta} \left(\sup_{|I| < \delta}\frac{1}{\omega(S(I))^\alpha}\int_{S(I)}|g'(z)|^2\omega^*(z)dA(z)\right.
\nonumber \\
&+& \left.\sup_{|I| < \delta}\frac{1}{\omega(S(I))^\alpha}\int_{S(I)}r^2|g'(rz)|^2\omega^*(z)dA(z)\right)
\nonumber \\
&=&\sup_{|I| < \delta}\left(\frac{1}{\omega(S(I))^\alpha}\int_{S(I)}|g'(z)|^2\omega^*(z)dA(z)\right)
\nonumber \\
&+& \sup_{r > 1 - \delta}\left(\sup_{|I| < \delta}\frac{1}{\omega(S(I))^\alpha}\int_{S(I)}r^2|g'(rz)|^2\omega^*(z)dA(z)\right).
\end{eqnarray}
Given an interval $I \subset \mathbb{T}$, let $e^{i\theta_0} \in I$ be the center point of $I$ and define a Carleson window $$S'(I) = \{re^{i\theta} \in \mathbb{D}: |\theta-\theta_0| < |I|, 1-2|I| \le r < 1\}.$$ Now $rS(I) \subset S'(I)$ for all $r \in ]1-\delta,1[,$ when $\delta$ is small enough. Also, it holds that $$\frac{\omega(S'(I))}{\omega(S(I))} \lesssim 1$$ for all $I \subset \mathbb{T}$ by the doubling property. Thus by the change of variables $u = rz$, we get
\begin{eqnarray*}
&&\frac{1}{\omega(S(I))^\alpha}\int_{S(I)}r^2|g'(rz)|^2\omega^*(z)dA(z)
\\
&=& \frac{1}{\omega(S(I))^\alpha}\int_{rS(I)}|g'(u)|^2\omega^*(u/r)dA(u)
\\
&\le& \left(\frac{\omega(S'(I))}{\omega(S(I))}\right)^\alpha \frac{1}{\omega(S'(I))^\alpha}\int_{S'(I)}|g'(u)|^2\omega^*(u)dA(u)
\\
&\lesssim& \frac{1}{\omega(S'(I))^\alpha}\int_{S'(I)}|g'(u)|^2\omega^*(u)dA(u)
\end{eqnarray*}
for all $r \in ]1-\delta,1[$ and consequently
\begin{eqnarray*}
&&\sup_{r > 1 - \delta}\left(\sup_{|I| < \delta}\frac{1}{\omega(S(I))^\alpha}\int_{S(I)}r^2|g'(rz)|^2\omega^*(z)dA(z)\right)
\\
&\lesssim& \sup_{|I| < \delta} \left(\frac{1}{\omega(S'(I))^\alpha}\int_{S'(I)}|g'(z)|^2\omega^*(z)dA(z)\right).
\end{eqnarray*}
Now the estimate \eqref{eq: distest1} becomes
\begin{eqnarray}
\label{eq: distest2}
\textup{dist}(g, C_0^\alpha(\omega^*))^2 &\lesssim& \sup_{|I| < \delta}\left(\frac{1}{\omega(S(I))^\alpha}\int_{S(I)}|g'(z)|^2\omega^*(z)dA(z)\right)
\nonumber \\
&+& \sup_{|I| < \delta} \left(\frac{1}{\omega(S'(I))^\alpha}\int_{S'(I)}|g'(z)|^2\omega^*(z)dA(z)\right).
\end{eqnarray}
Letting $\delta \to 0^+$ in \eqref{eq: distest2}, we get
$$\textup{dist}(g, C_0^\alpha(\omega^*))^2 \lesssim \limsup_{|I| \to 0}\left(\frac{1}{\omega(S(I))^\alpha}\int_{S(I)}|g'(z)|^2\omega^*(z)dA(z)\right).$$ The proof is complete.

\end{proof}

\section{ Norm Estimate}\label{s2}
Define  \begin{equation}\label{fap}f_{a,p}(z)=\frac{(1-|a|)^{\frac{\gamma+1}{p}}}{(1-\bar{a}z)^{\frac{\gamma+1}{p}}\omega(S(a))^{\frac{1}{p}}},\end{equation} where $\gamma=\beta(\omega)>0$ is the constant in Lemma 1.1 \cite{PR}. A simple computation shows that $\sup_{a\in \mathbb D}\|f_{a,p}\|_{A_\omega^p}\lesssim 1$, and $f_a(z)\to 0$ uniformly on compact subsets of $\mathbb D$ as $|a|\to 1$.

\begin{lemma}\label{Tgfa}
Let $0< p\leq q < \infty$, $s = 2\left(\frac{1}{p}-\frac{1}{q}\right)+1$, $\frac{1}{p}-\frac{1}{q}<1$, $\omega\in \widehat{D}$ and  $g\in \mathcal{C}^s(\omega^*)$. Then
\begin{eqnarray}\label{Tgfap}\limsup_{|a|\to 1}\|T_g(f_{a,p})\|_{A_\omega^q}\geq \limsup_{|a|\to 1}\left(\frac{\int_{S(a)}|g'(z)|^2\omega^*(z)dA(z)}{\omega(S(a))^s}\right)^{\frac{1}{2}}\end{eqnarray}
\end{lemma}

\begin{proof}
We split the analysis into two cases.

{\it Case 1}- Assume $p=q$. For this, we divide the proof of the claim (\ref{Tgfap}) into three sub-cases.

{\it Sub-case 1}: $p>2$.
We may assume that $|a| > 1/2$.
For $z\in S(a)$, it is easy to see that $f_{a,p}(z)\simeq \omega(S(a))^{-\frac{1}{p}}$, and so
\begin{equation}\label{eq: el1}
%&&\int_{\mathbb D}|f_{a,p}|^p|g'(z)|^2\omega^*(z) dA(z)\nonumber\\
%&&\gtrsim
\int_{S(a)}|f_{a,p}|^p|g'(z)|^2\omega^*(z) dA(z)
\gtrsim \frac{1}{\omega(S(a))}\int_{S(a)}|g'(z)|^2\omega^*(z)dA(z).
\end{equation}
Furthermore, by applying Fubini's theorem, H\"{o}lder's inequality, Lemma \ref{N(f)} and (4),  we obtain
\begin{eqnarray*}
&&\int_{S(a)}|f_{a,p}|^p|g'(z)|^2\omega^*(z)dA(z)\\
&&\lesssim \int_{\mathbb D}|f_{a,p}|^p|g'(z)|^2\int_{T(z)}\omega(u)dA(u)dA(z)\\
&&=\int_{\mathbb D}\int_{\Gamma(u)}|f_{a,p}(z)|^p|g'(z)|^2dA(z)\omega(u)dA(u)\\
&&\leq\int_{\mathbb D} N(f_{a,p})(u)^{p-2}\int_{\Gamma(u)}|f_{a,p}(z)|^2|g'(z)|^2dA(z)w(u)dA(u)\\
&&\leq \left(\int_{\mathbb D}N(f_{a,p})(u)^p\omega(u)dA(u)\right)^{\frac{p-2}{p}}\left(\int_{\mathbb D}\left(\int_{\Gamma(u)}|f_{a,p}(z)|^2|g'(z)|^2dA(z)\right)^{\frac{p}{2}}\omega(u)dA(u)\right)^{\frac{2}{p}}\\
&&\lesssim \|f_{a,p}\|_{A_\omega^p}^{p-2}\|T_g(f_{a,p})\|_{A_\omega^p}^2\lesssim \|T_g(f_{a,p})\|_{A_\omega^p}^2.
\end{eqnarray*}
This last estimate, along with \eqref{eq: el1} gives $$\|T_gf_{a,p}\|_{A_\omega^p}^2\gtrsim \frac{1}{\omega(S(a))}\int_{S(a)}|g'(z)|^2\omega^*(z)dA(z).$$

{\it Sub-case 2}: $p=2$. The desired estimate follows from (\ref{equiv3}) immediately.

{\it Sub-case 3}: $0< p<2$. Let $1<\alpha,\beta<\infty$ be such that $\beta/\alpha=p/2<1$, $|a| > 1/2$ and let $\alpha'$ and $\beta'$ be the conjugate indexes of $\alpha$ and $\beta$ respectively. It follows from Fubini's theorem , H\"{o}lder's inequality, and (\ref{equiv2}) that
\begin{eqnarray}\label{necep2}
&&\frac{1}{\omega(S(a))^{\frac{2}{p}}}\int_{S(a)}|g'(z)|^2\omega^*(z)dA(z)\nonumber\\
&& \simeq\int_{S(a)}|g'(z)|^2|f_{a,p}(z)|^2\omega^*(z)dA(z)\nonumber\\
&& \simeq \int_{\mathbb D}\left(\int_{S(a)\cap\Gamma(u)}|g'(z)|^2|f_{a,p}(z)|^2dA(z)\right)^{\frac{1}{\alpha}+\frac{1}{\alpha'}}\omega(u)dA(u)\nonumber\\
&& \leq \left(\int_{\mathbb D}\left(\int_{\Gamma(u)}|g'(z)|^2|f_{a,p}(z)|^2dA(z)\right)^{\frac{\beta}{\alpha}}\omega(u)dA(u)\right)^{\frac{1}{\beta}}\\
&& \ \ \cdot \left(\int_{\mathbb D}\left(\int_{\Gamma(u)\cap S(a)}|g'(z)|^2|f_{a,p}(z)|^2dA(z)\right)^{\frac{\beta'}{\alpha'}}\omega(u)dA(u)\right)^{\frac{1}{\beta'}}\nonumber\\
&&\simeq \|T_g(f_{a,p})\|_{A_\omega^p}^{\frac{p}{\beta}} \omega(S(a))^{-\frac{2}{p\alpha'}}\|S_g(\chi_{S(a)})\|_{L_\omega^{\frac{\beta'}{\alpha'}}}^{\frac{1}{\alpha'}}  \nonumber
\end{eqnarray}
for $|a| > 1/2,$ %$a\in \mathbb D$,
where $$S_g(\psi)(u)=\int_{\Gamma(u)}|\psi(z)|^2|g'(z)|^2dA(z), \quad u\in \mathbb D\setminus \{0\},$$ for any bounded function $\psi$ on $\mathbb D$. From $1<\beta<\alpha$, we obtain $\frac{\beta'}{\alpha'}>1$ with the conjugate exponent $\left(\frac{\beta'}{\alpha'}\right)'=\frac{\beta(\alpha-1)}{\alpha-\beta}>1$. Thereby
\begin{equation}
\label{Sg}
\|S_g(\chi_{S(a)})\|_{L_\omega^{\frac{\beta'}{\alpha'}}}=\sup_{\|f\|_{L_\omega^{\frac{\beta(\alpha-1)}{\alpha-\beta}}}\leq 1}\left|\int_{\mathbb D}f(u)S_g(\chi_{S(a)})(u)\omega(u)dA(u)\right|.
\end{equation}
Combining Fubini's theorem, H\"{o}lder's inequality, and Lemma \ref{eCarle}, we conclude that
\begin{eqnarray*}
&& \left|\int_{\mathbb D}f(u)S_g(\chi_{S(a)})(u)\omega(u)dA(u)\right|\\
&&\leq \int_{\mathbb D}|f(u)|\int_{\Gamma(u)\cap S(a)}|g'(z)|^2dA(z)\omega(u)dA(u)\\
&& =\int_{S(a)}|g'(z)|^2\int_{T(z)}|f(u)|\omega(u)dA(u)dA(z)\\
&& \lesssim \int_{S(a)}M_\omega(|f|)(z)|g'(z)|^2\omega^*(z)dA(z)\\
&&\leq \left(\int_{S(a)}|g'(z)|^2\omega^*(z)dA(z)\right)^{\frac{\alpha'}{\beta'}}
\left(\int_{S(a)}M_\omega(|f|)(z)^{\left(\frac{\beta'}{\alpha'}\right)'}|g'(z)|^2\omega^*(z)dA(z)\right)^{1-\frac{\alpha'}{\beta'}}\\
&& \leq \left(\int_{S(a)}|g'(z)|^2\omega^*(z)dA(z)\right)^{\frac{\alpha'}{\beta'}}\left(\sup_{b\in \mathbb D}\frac{\mu_a(S(b))}{\omega(S(b))}\right)^{1-\frac{\alpha'}{\beta'}}\|f\|_{L_\omega^{\left(\frac{\beta'}{\alpha'}\right)'}},
\end{eqnarray*} where $d\mu_a(z)=\chi_{S(a)}(z)|g'(z)|^2\omega^*(z)dA(z)$. The last estimate, along with (\ref{necep2}) and (\ref{Sg})
 gives
 \begin{eqnarray*}
 &&\frac{\int_{S(a)}|g'(z)|^2\omega^*(z)dA(z)}{\omega(S(a))^{\frac{2}{p}}}\lesssim \|T_g(f_{a,p})\|_{A_\omega^p}^{\frac{p}{\beta}}\\
 &&\cdot\frac{\left(\int_{S(a)}|g'(z)|^2\omega^*(z)dA(z)\right)^{\frac{1}{\beta'}}}
 {\omega(S(a))^{\frac{2}{p}\cdot\frac{1}{\alpha'}}}\left(\sup_{b\in \mathbb D}\frac{\mu_a(S(b))}{\omega(S(b))}\right)^{(1-\frac{\alpha'}{\beta'})\cdot\frac{1}{\alpha'}},
 \end{eqnarray*} so that
 \begin{eqnarray*}
 \left(\frac{\int_{S(a)}|g'(z)|^2\omega^*(z)dA(z)}{\omega(S(a))}\right)^{\frac{1}{\beta}}\lesssim \|T_g(f_{a,p})\|_{A_\omega^p}^{\frac{p}{\beta}}\left(\sup_{b\in \mathbb D}\frac{\mu_a(S(b))}{\omega(S(b))}\right)^{\frac{1}{\beta}(1-\frac{\beta}{\alpha})},
 \end{eqnarray*} from which we obtain
 \begin{eqnarray}\label{necep1}&&\frac{\int_{S(a)}|g'(z)|^2\omega^*(z)dA(z)}{\omega(S(a))}\lesssim \|T_g(f_{a,p})\|_{A_\omega^p}^p\left(\sup_{b\in \mathbb D}\frac{\mu_a(S(b))}{\omega(S(b))}\right)^{1-\frac{p}{2}}\nonumber\\
 &&= \|T_g(f_{a,p})\|_{A_\omega^p}^p\left(\sup_{b: S(b)\subseteq S(a)}\frac{\mu_a(S(b))}{\omega(S(b))}\right)^{1-\frac{p}{2}}.
 \end{eqnarray}  It is easy to see that
 \begin{eqnarray*}
 &&\limsup_{|a|\to 1} \sup_{b: S(b)\subseteq S(a)}\frac{\mu_a(S(b))}{\omega(S(b))}=\limsup_{|a|\to 1} \sup_{b: S(b)\subseteq S(a)}\frac{\int_{S(b)}|g'(z)|^2\omega^*(z)dA(z)}{\omega(S(b))}\\
 &&=\limsup_{|a|\to 1}\frac{\int_{S(a)}|g'(z)|^2\omega^*(z)dA(z)}{\omega(S(a))}.
 \end{eqnarray*}
 The last equality and (\ref{necep1}) yield
 \begin{eqnarray}\label{eq: limsupestimate}&&\left(\limsup_{|a|\to 1}\frac{\int_{S(a)}|g'(z)|^2\omega^*(z)dA(z)}{\omega(S(a))}\right)^{\frac{p}{2}} \nonumber \\
 &&=\limsup_{|a|\to 1}\frac{\frac{\int_{S(a)}|g'(z)|^2\omega^*(z)dA(z)}{\omega(S(a))}}{\left(\sup_{b: S(b)\subseteq S(a)}\frac{\int_{S(a)}|g'(z)|^2\omega^*(z)dA(z)}{\omega(S(a))}\right)^{1-\frac{p}{2}}}
 \lesssim \limsup_{|a|\to 1}\|T_g(f_{a,p})\|_{A_\omega^p}^p.\end{eqnarray} Therefore
 $$\limsup_{|a|\to 1}\|T_g(f_{a,p})\|_{A_\omega^p}
 \gtrsim \limsup_{|a|\to 1}\left(\frac{\int_{S(a)}|g'(z)|^2\omega^*(z)dA(z)}{\omega(S(a))}\right)^{\frac{1}{2}}.$$
{\it Case 2}- Assume $p<q$.

 For all $h\in A_\omega^q$, we have
$$\|h\|_{A_\omega^q}^q\geq \int_{\mathbb D\setminus D(0,r)}|h(z)|^p \omega(z)dA(z)\gtrsim M_p^p(r,h)\int_r^1\omega(s)ds,\quad r\geq \frac{1}{2},$$
where $M_p(r,h) = \left( \frac{1}{2\pi} \int_0^{2\pi}|h(re^{i\theta})|^p d\theta\right)^{1/p}$. Then $$M_q^q(r,T_gf_{a,p})\lesssim \frac{\|T_gf_{a,p}\|_{A_\omega^q}^q}{\int_r^1\omega(s)ds},\quad r\geq \frac{1}{2}.$$
By Cauchy's integral formula, we  get two  well-known estimates as $M_\infty(r,f')\lesssim M_\infty(\rho,f)/(1-r)$ and $M_\infty(r,f)\lesssim M_q(\rho,f)(1-r)^{-1/q}$, $\rho=(1+r)/2$.
Then
\begin{eqnarray*}
&&|g'(a)|\simeq \omega^*(a)^{\frac{1}{p}}|(T_gf_{a,p})'(a)|
\lesssim \omega^*(a)^{\frac{1}{p}}\frac{M_\infty\left(\frac{1+|a|}{2},T_gf_{a,p}\right)}{1-|a|}\\
&&\lesssim \omega^*(a)^{\frac{1}{p}}\frac{M_q\left(\frac{3+|a|}{4},T_gf_{a,p}\right)}{(1-|a|)^{1+\frac{1}{q}}}
\lesssim \omega^*(a)^{\frac{1}{p}}\frac{\|T_gf_{a,p}\|_{A_\omega^q}}{(1-|a|)^{1+\frac{1}{q}}\left(\int_{\frac{3+|a|}{4}}^1\omega(s)ds\right)^{\frac{1}{q}}}\\
&&\simeq \frac{\omega^*(a)^{\frac{1}{p}-\frac{1}{q}}\|T_gf_{a,p}\|_{A_\omega^q}}{1-|a|}.
\end{eqnarray*}
The last inequality is due to $\omega^*(a)\simeq (1-|a|)\int_{|a|}^1\omega(s)ds\lesssim (1-|a|)\int_{\frac{3+|a|}{4}}^1\omega(s)ds$, for $|a|>\frac{1}{2}$.
Thus $$\|T_gf_{a,p}\|_{A_\omega^q}\gtrsim |g'(a)|\omega^*(a)^{\frac{1}{q}-\frac{1}{p}}(1-|a|), |a|>\frac{1}{2}$$ and so  Lemma \ref{littleg} yields (\ref{Tgfap}).

\end{proof}

\begin{proof}[Proof of Theorem \ref{norm}]

Clearly, $\|g\|_{*,\alpha,\omega}\simeq B$ and $\|g\|_{*,\alpha,\omega}\simeq C$ follow by the proof of Lemma 5.3 and Proposition 4.7 in \cite{PR},
respectively. The proof of Lemma \ref{Tgfa} also deduces that $\sup_{a\in \mathbb D}\|T_gf_{a,p}\|_{A_\omega^q}\gtrsim \|g\|_{*,\alpha,\omega} $. So $\|T_g\|_{A_\omega^p\rightarrow A_\omega^q}\gtrsim \|g\|_{*,\alpha,\omega}$. It remains to prove $\|T_g\|_{A_\omega^p\rightarrow A_\omega^q}\lesssim \|g\|_{*,\alpha,\omega}$.

Notice
\begin{equation}\label{e1}\|T_g\|_{A_\omega^p\rightarrow A_\omega^q}^q=\sup_{\|f\|_{A_\omega^p}\leq 1}\|T_g(f)\|_{A_\omega^q}^q.\end{equation}
Two cases have to be analyzed.

{\it Case 1}-Assume $q\geq 2$. Applying  (\ref{equiv1}) and  H\"{o}lder's inequality,  for $q>2$ we get
\begin{eqnarray*}
&&\|T_gf\|_{A_\omega^q}^q\simeq \int_{\mathbb D}|T_gf(z)|^{q-2}|f(z)|^2|g'(z)|^2\omega^*(z)dA(z)\\
&&\leq\left(\int_{\mathbb D}|T_gf(z)|^{\frac{2q-2p+pq}{p}}|g'(z)|^2\omega^*(z)dA(z)\right)^{\frac{p(q-2)}{2q-2p+pq}}\\
&&\ \ \ \ \cdot\left(\int_{\mathbb D}|f(z)|^{\frac{2q-2p+pq}{q}}|g'(z)|^2\omega^*(z)dA(z)\right)^{\frac{2q}{2q-2p+pq}},
\end{eqnarray*}
whence \begin{equation}\label{q>2}
\|T_gf\|_{A_\omega^q}^q\lesssim \mathsf{U}^{\frac{p(q-2)}{2q-2p+pq}}\mathsf{V}^{\frac{2q}{2q-2p+pq}},
\end{equation}
where \begin{equation}\label{UV}
\begin{cases}
\mathsf{U}=\int_{\mathbb D}|T_gf(z)|^\frac{2q-2p+pq}{p}|g'(z)|^2\omega^*(z)dA(z);\\
\mathsf{V}=\int_{\mathbb D}|f(z)|^\frac{2q-2p+pq}{q}|g'(z)|^2\omega^*(z)dA(z)
\end{cases}
\end{equation}
 Noticing that (\ref{q>2}) is also true for $q=2$. We have to control $\mathsf{U}$ and $\mathsf{V}$ from above. To do so, set $d\mu(z)=|g'(z)|^2\omega^*(z)dA(z)$.
The assumption $g\in \mathcal C^\alpha(\omega^*)$  and Theorem \ref{Tbound} yield the boundedness of $T_g: A_\omega^p \to A_\omega^q$. Moreover, Lemma \ref{eCarle} and the fact that $\alpha=2\left(\frac{1}{p}-\frac{1}{q}\right)+1=\left(\frac{2q-2p+pq}{p}\right)/q=\left(\frac{2q-2p+pq}{q}\right)/p$ ensure
$d\mu$ is a $\frac{2q-2p+pq}{p}$-Carleson measure for $A_\omega^q$ and also a $\frac{2q-2p+pq}{q}$-Carleson measure for $A_\omega^p$.
Consequently,  the inequality (\ref{eCarnorm}) is applied to deduce that
\begin{eqnarray}\label{eu}\mathsf{U}\lesssim \left(\sup_{I\subseteq \mathbb T}\frac{\mu(S(I))}{\omega(S(I))^\alpha}\right)\|T_gf\|_{A_\omega^q}^{\frac{2q-2p+pq}{p}},
\end{eqnarray}
and
\begin{equation}\label{ev}\mathsf{V}\lesssim \left(\sup_{I \subseteq \mathbb T }\frac{\mu(S(I))}{\omega(S(I))^\alpha}\right)\|f\|_{A_\omega^p}^{\frac{2q-2p+pq}{q}}.\end{equation}
A combination of (\ref{q>2}), (\ref{eu}) and (\ref{ev}) gives when $n\to \infty$ that
\begin{eqnarray*}
&&\|T_gf\|_{A_\omega^q}^q\lesssim \left(\sup_{I\subseteq \mathbb T}\frac{\mu(S(I))}{\omega(S(I))^\alpha}\right)^{\frac{p(q-2)}{2q-2p+pq}}\|T_gf\|_{A_\omega^q}^{q-2}\\
&&\ \ \ \ \ \ \ \ \cdot\left(\sup_{I \subseteq \mathbb T}\frac{\mu(S(I))}{\omega(S(I))^\alpha}\right)^{\frac{2q}{2q-2p+pq}}\|f\|_{A_\omega^p}^2\\
&&=\left(\sup_{I \subseteq \mathbb T}\frac{\mu(S(I))}{\omega(S(I))^\alpha}\right) \|T_gf\|_{A_\omega^q}^{q-2}\|f\|_{A_\omega^p}^2.
\end{eqnarray*}
It follows that $$\|T_g\|_{A_\omega^p\rightarrow A_\omega^q}=\sup_{\|f\|_{A_\omega^p}\leq 1}\|T_gf\|_{A_\omega^q}\lesssim \left(\sup_{I \subseteq\mathbb T}\frac{\mu(S(I))}{\omega(S(I))^\alpha}\right)^{\frac{1}{2}}.$$

{\it Case 2}-Assume $0< p\leq q<2$. From the equation (\ref{equiv2}), H\"{o}lder's inequality,   Fubini's theorem and Lemma \ref{N(f)} it follows that
\begin{eqnarray*}
&&\|T_gf\|_{A_\omega^q}^q\simeq \int_{\mathbb D}\left(\int_{\Gamma(u)}|f(z)|^2|g'(z)|^2dA(z)\right)^{\frac{q}{2}}\omega(u)dA(u)\\
&&\leq\int_{\mathbb D}N(f)(u)^{\frac{p(2-q)}{2}}\left(\int_{\Gamma(u)}|f(z)|^{2-\frac{2p}{q}+p}|g'(z)|^2dA(z)\right)^{\frac{q}{2}}\omega(u)dA(u)\\
&&\leq\left(\int_{\mathbb D}N(f)(u)^p\omega(u)dA(u)\right)^{\frac{2-q}{2}}\\
&&\cdot
\left(\int_{\mathbb D}\int_{\Gamma(u)}|f(z)|^{2-\frac{2p}{q}+p}|g'(z)|^2dA(z)\omega(u)dA(u)\right)^{\frac{q}{2}}\\
&&= \|N(f)\|_{A_\omega^p}^{\frac{p(2-q)}{2}}\left(\int_{\mathbb D}|f(z)|^{2-\frac{2p}{q}+p}|g'(z)|^2\omega(T(z))dA(z)\right)^{\frac{q}{2}}\\
&&\simeq\|f\|_{A_\omega^p}^{\frac{p(2-q)}{2}}\mathsf{V}^{\frac{q}{2}},
\end{eqnarray*} where $\mathsf{V}$ is defined in (\ref{UV}). %Lemma \ref{N(f)} and
Now (\ref{ev}) ensures that
\begin{eqnarray*}&&\|T_gf\|_{A_\omega^q}^q\lesssim \|f\|_{A_\omega^p}^{\frac{p(2-q)}{2}}\left(\left(\sup_{I\subseteq \mathbb T}\frac{\mu(S(I))}{\omega(S(I))^\alpha}\right)\|f\|_{A_\omega^p}^{\frac{2q-2p+pq}{q}}\right)^{\frac{q}{2}}\\
&&= \left(\sup_{I\subseteq\mathbb T}\frac{\mu(S(I))}{\omega(S(I))^\alpha}\right)^{\frac{q}{2}}\|f\|_{A_\omega^p}^q. \end{eqnarray*} Consequently we get that
 $$\sup_{\|f\|_{A_\omega^p}\leq 1}\|T_gf\|_{A_\omega^q}\lesssim \left(\sup_{I\subseteq \mathbb T}\frac{\mu(S(I))}{\omega(S(I))^\alpha}\right)^{\frac{1}{2}},$$ and then
$$\|T_g\|_{A_\omega^p\rightarrow A_\omega^q}\lesssim \|g\|_{*,\alpha,\omega}.$$
The proof is complete.
\end{proof}

\section{ Essential Norm and Weak Compactness}\label{s2}
\begin{proof}[Proof of Theorem \ref{essnorm}]
Lemma 5.3 in \cite{PR} and Lemma \ref{littleg} show that $B\simeq C$ and $B\simeq D$. $A\simeq B$ follows by Lemma \ref{distance}.

To prove $\|T_g\|_{e,A_\omega^p\rightarrow A_\omega^q}\lesssim A$, observe that for any $h\in C_0^\alpha(\omega^*)$, by Theorem \ref{norm},
$$\|T_g\|_{e,A_\omega^p\rightarrow A_\omega^q}\lesssim \|T_g-T_h\|_{A_\omega^p\rightarrow A_\omega^q}=\|T_{g-h}\|_{A_\omega^p\rightarrow A_\omega^q}\lesssim \|g-h\|_{C^\alpha(\omega^*)},$$
whence, taking infimum over $h$, we obtain
$\|T_g\|_{e,A_\omega^p\rightarrow A_\omega^q}\lesssim A$.

Next, we turn to establishing
\begin{equation}\label{lower}\|T_g\|_{e,A_\omega^p\rightarrow A_\omega^q}\gtrsim B=\left(\limsup_{|I|\to0}\frac{\int_{S(I)}|g'(z)|^2\omega^*(z)dA(z)}{\omega(S(I))^\alpha}\right)^{1/2}.\end{equation}
Given a subarc $I$ of $\mathbb T$, consider $f_{a,p}$ defined in (\ref{fap}),
where $a=(1-|I|)\zeta$ and $\zeta$ is the center point of $I$. Then $S(a)=S(I)$.
 For the moment fix
a compact operator $K: A_\omega^p\to A_\omega^q$. Then
$$\lim_{|a|\to 1}\|Kf_{a,p}\|_{A_\omega^q}=0,$$
and so we find that \begin{eqnarray*}
&&\|T_g-K\|\gtrsim \limsup_{|a|\to 1}\|(T_g-K)f_{a,p}\|_{A_\omega^q}\\
&&\gtrsim \limsup_{|a|\to1}(\|T_gf_{a,p}\|_{A_\omega^q}-\|Kf_{a,p}\|_{A_\omega^q})\\
&&\gtrsim \limsup_{|a|\to 1}\|T_gf_{a,p}\|_{A_\omega^q}.
\end{eqnarray*}
Upon taking the infimum of both sides  of this inequality over all compact operators $K: A_\omega^p\to A_\omega^q$, it follows from Lemma \ref{Tgfa} that
\begin{equation}\label{ele}\|T_g\|_{e,A_\omega^p\rightarrow A_\omega^q}\gtrsim \limsup_{|a|\to 1}\|T_gf_{a,p}\|_{A_\omega^q}\gtrsim B.\end{equation}
The proof is complete.
\end{proof}

Next, we consider the weak essential norm of $T_g$ on $A^1_\omega$. Recall that the notion of weak compactness of an operator is non-trivial only on non-reflexive spaces. The non-reflexivity of $A^1_\omega$ can be shown e.g.\ by constructing an isomorphic copy of the sequence space $\ell^1$ inside $A^1_\omega$. For this one uses suitable normalized functions so that the closed subspace spanned by these functions is isomorphic to $\ell^1.$ One may use e.g.\ functions
\begin{displaymath}
f_{r_k, \gamma}(z) = \frac{g_{r_k, \gamma}(z)}{\| g_{r_k, \gamma} \|_{A^1_\omega}}, \quad z \in \mathbb{D},
\end{displaymath}
where $r_k \in (0,1),\, r_k \to 1$ sufficiently fast, $g_{r_k, \gamma}(z) = \left( \frac{1-r_k}{1-r_k z}\right)^\gamma$ and $\gamma > 0$. The functions $f_{r_k, \gamma}$ have the properties
\begin{enumerate}
\item[(i)] $\int_{\mathbb{D} \setminus D(1,\varepsilon)} |f_{r_k, \gamma}|\omega dA \to 0,$ as $k \to \infty$ \textup{ for all $\varepsilon > 0$;}

\item[(ii)] $\int_{\mathbb{D} \cap D(1,\delta)} |f_{r_k, \gamma}|\omega dA \to 0,$ as $\delta \to 0$ \textup{ for all $k = 1,2,\ldots$.}
\end{enumerate}
The condition (i) follows from the doubling property of $\widehat{\omega}$ and the choice for the parameter $\gamma$ to be large enough. The condition (ii) is evident. These properties and the fact that $r_k \to 1$ sufficiently fast ensure that the map
\begin{displaymath}
U \colon \ell^1 \to A^1_\omega,\quad U((\alpha_k)_{k = 1}^\infty) = \sum_{k = 1}^\infty \alpha_k f_{r_k, \gamma}
\end{displaymath}
is an isomorphism onto its image.

In order to deal with the weak essential norm of $T_g$ on $A^1_\omega$ we utilize the classical Dunford-Pettis criterion (see e.g.\ \cite[Theorem 5.2.9]{AK}), which states that a bounded set $S \subset L^1_\mu,$ (where the measure $\mu$ is a probability measure) is relatively compact in the weak topology of  $L^1_\mu$ if and only if it is equi-integrable, i.e.,
\begin{displaymath}
\lim_{\mu(A) \to 0}\sup_{f \in S} \int_A |f| d\mu = 0.
\end{displaymath}
The application of this criterion in our setting is based on the next lemma.

\begin{lemma} \label{integrallimit}
 Let $\omega\in \widehat{D}$. Suppose $g\in C^1(\omega^*)$.
For all non-zero $a\in \mathbb D$, let $J(a)=\{re^{i\theta}: |\theta-\arg a|<(1-|a|)^{1/6},1-|a|<r<1\}$ and
\[
f_a(z)=f_{a,1}(z)=\frac{(1-|a|^2)^{\gamma+1}}{(1-\overline{a}z)^{\gamma+1}\omega(S(a))},
\]
where $\gamma$ is large enough so that
$$\lim_{|a|\to 1}\frac{(1-|a|)^{\frac{5}{6}\gamma-\frac{1}{6}}}{\int_{|a|}^1\omega(s)ds}=0.$$ Then
$$\lim_{|a|\to 1}\int_{\mathbb D\setminus J(a)}|T_gf_a(z)|\omega(z)dA(z)=0.$$
\end{lemma}

\begin{proof}
We may assume that $g(0)=0$ and $0<a<1$ due to rotation invariance. It is not hard to see that  for all $0\leq r<1$ and $|\theta|\leq \pi$,
 we have $|1-are^{i\theta}|\geq c|\theta|$, where $c>0$ is an absolute constant.  For $z\in \mathbb D\setminus J(a)$ and $a>\frac{1}{2}$, we have
 $$|f_a(z)|\lesssim \frac{(1-a)^{\gamma+1}}{|\theta|^{\gamma+1}\omega(S(a))}\leq \frac{(1-a)^{\frac{5}{6}(\gamma+1)}}{\omega(S(a))}\simeq\frac{(1-a)^{\frac{5}{6}(\gamma+1)}}
 {(1-a)\int_{|a|}^1\omega(s)ds}=\frac{(1-a)^{\frac{5}{6}\gamma-\frac{1}{6}}}{\int_{|a|}^1\omega(s)ds}.$$
 Proposition 5.1 in \cite{PR} and $g\in C^1(\omega^*)$ imply that $g\in A_\omega^1$.
Therefore,  by using equation (\ref{equiv2}) twice we obtain
\begin{eqnarray*}
&&\int_{\mathbb D\setminus J(a)}|T_gf_a(z)|\omega(z)dA(z)\\
&&\simeq \int_{\mathbb D}\left(\int_{\Gamma(u)\setminus J(a)}|f_a(z)|^2|g'(z)|^2dA(z)\right)^{\frac{1}{2}}\omega(u)dA(u)\\
&&\lesssim \frac{(1-a)^{\frac{5}{6}\gamma-\frac{1}{6}}}{\int_{|a|}^1\omega(s)ds}\|g\|_{A_\omega^1},
\end{eqnarray*} which tends to 0 as $|a|\to 1$. The proof is complete.
\end{proof}

\begin{proof}[Proof of Theorem \ref{theorem2}]
Since compact operators are also weakly compact, we have
\begin{displaymath}
\lVert T_g \rVert_{w, A^1_\omega \to A^1_\omega} \le \lVert T_g \rVert_{e, A^1_\omega \to A^1_\omega} \lesssim \limsup_{|I|\to0}\left(\frac{\int_{S(I)}|g'(z)|^2\omega^*(z)dA(z)}{\omega(S(I))}\right)^{1/2}.
\end{displaymath}
To verify the lower estimate for $\|T_g\|_{w,A_\omega^1\rightarrow A_\omega^1}$, suppose that $W\colon A^1_\omega \to A^1_\omega$ is any weakly compact operator. Assume that $J(a)$ and $f_a$ satisfy the conditions from Lemma \ref{integrallimit}.
As $\|f_a\|_{A_\omega^1}\simeq 1$, we have
\begin{eqnarray}
\label{eq: estimate21}
&& \| T_g - W \| \gtrsim \| (T_g - W)f_a \|_{A^1_\omega} = \int_{\mathbb{D}}|T_g f_a(z) - W f_a(z)|\omega(z) dA(z)
\nonumber \\
&&\ge \int_{J(a)}|T_g f_a(z) - W f_a(z)|\omega(z) dA(z)\\
&& \ge \int_{J(a)} |T_g f_a(z)| \omega(z) dA(z) -  \int_{J(a)} |W f_a(z)| \omega(z) dA(z).\nonumber
\end{eqnarray}
Now by Lemma \ref{integrallimit}
\begin{displaymath}
\lim_{|a| \to 1} \int_{\mathbb{D}\setminus J(a)} |T_g f_a(z)| \omega(z) dA(z) = 0,
\end{displaymath}
and therefore
\begin{equation}
\label{eq: limit11}
\limsup_{|a| \to 1} \int_{J(a)} |T_g f_a(z)|\omega(z) dA(z) = \limsup_{|a| \to 1} \lVert T_g f_a \rVert_{A^1_\omega}.
\end{equation}
Since the set $\{Wf_a \colon a \in \mathbb{D}\}$ is relatively weakly compact in $A^1_\omega$ and hence equi-integrable in $L^1_\omega$, it holds that
\begin{equation}
\label{eq: limit21}
\lim_{|a| \to 1} \int_{J(a)} |W f_a(z)| \omega(z) dA(z) = 0
\end{equation}
by the Dunford-Pettis criterion. Now taking $\limsup_{|a| \to 1}$ in \eqref{eq: estimate21} we have
\begin{displaymath}
\lVert T_g - W \rVert \ge \limsup_{|a| \to 1} \lVert T_g f_a \rVert_{A^1_\omega}
\end{displaymath}
by \eqref{eq: limit11} and \eqref{eq: limit21}. Hence by taking infimum over all weakly compact operators $W\colon A^1_\omega \to A^1_\omega$ we get
\begin{displaymath}
\lVert T_g \rVert_{w, A^1_\omega \to A^1_\omega} \ge \limsup_{|a| \to 1} \lVert T_g f_a \rVert_{A^1_\omega}.
\end{displaymath}
%Finally, the estimate \eqref{eq: limsupestimate} for $p=1$ states that
Finally, it follows from Lemma \ref{Tgfa} that
\begin{displaymath}
 \limsup_{|a| \to 1} \lVert T_g f_a \rVert_{A^1_\omega} \gtrsim \limsup_{|I|\to 0}\left(\frac{\int_{S(I)}|g'(z)|^2\omega^*(z)dA(z)}{\omega(S(I))}\right)^{1/2}.
\end{displaymath} The last inequality, along with Lemma \ref{distance} completes the proof.
\end{proof}


\begin{thebibliography}{XX}
\bibitem{AK} F. Albiac and N. Kalton, Topics in Banach Space Theory, Springer-Verlag, New York, 2006.
\bibitem{B} O. Blasco, S. D\'{i}az-Madrigal, J. Mart\'{i}nez, M- Papadimitrakis and A. Siskakis, Semigroups of composition operators and integral operators in spaces of analytic functions, Ann. Acad. Sci. Fenn. Math. {\bf 38} (2013), no. 1, 67-89.
\bibitem{D} P.L. Duren, Theory of $H^p$ Spaces, Dover Publications, New York, 2000.
\bibitem{HL} F. Hirsch and G. Lacombe,  Elements of Functional Analysis,  Graduate Texts in Mathematics, 192. Springer-Verlag, New York, 1999.
%\bibitem{K}  Y. Katznelson,  An introduction to harmonic analysis (Third edition), Cambridge Mathematical Library, Cambridge University Press, Cambridge, 2004.
\bibitem{LST} B. Liu, E. Saksman, H-O. Tylli, Small composition operators on analytic vector-valued function spaces, Pacific J. Math. {\bf 184} (1998), no. 2, 295–-309.
\bibitem{LLX} J. Liu, Z. Lou and C. Xiong, Essential norms of integral operators on spaces of analytic functions,  Nonlinear Anal. {\bf 75} (2012), 5145--5156.

\bibitem{LMN} J. Laitila, S. Miihkinen, P.J. Nieminen, Essential norms and weak compactness of integration operators,  Arch. Math. {\bf 97} (2011), 39--48.
%\bibitem{P} M. Pavlovi\'{c}, Introduction to Function Spaces on the Disk, in: Posebna Izdanja (Special Editions), vol. 20, Matemati\v{c}ki Institut SANU, Beograd, 2004.
\bibitem{PR2} J.A. Pel\'{a}ez, J. R\"{a}tty\"{a}, Embedding theorems for Bergman spaces via harmonic analysis, Math. Ann. DOI 10.1007/s00208-014-1108-5.
\bibitem{PR1} J.A. Pel\'{a}ez, J. R\"{a}tty\"{a}, Generalized Hilbert operators on weighted Bergman spaces,  Adv. Math. {\bf 240} (2013), 227--267.
\bibitem{PR3}  J.A. Pel\'{a}ez, J. R\"{a}tty\"{a}, Trace class criteria for Toeplitz and composition operators on small Bergman spaces, arXiv:1501.00131 [math.FA]
\bibitem{PR} J.A. Pel\'{a}ez, J. R\"{a}tty\"{a}, Weighted Bergman spaces induced by rapidly increasing weights,  Mem. Amer. Math. Soc. {\bf 227} (2014), no. 1066.
\bibitem{R} J. R\"{a}tty\"{a}, Integration operator acting on Hardy  and weighted Bergman spaces,  Bull. Austral. Math. Soc. {\bf 75} (2007), 431--446.
\bibitem{SS} J.H. Shapiro, C. Sundberg, Compact composition operators on $L^1$, Proc. Amer. Math. Soc. {\bf 108} (1990), no. 2, 443-–449.

\bibitem{S} A. Siskakis, Weighted integrals of analytic functions,  Acta Sci. Math. (Szeged) {\bf 66} (2000), no. 3-4, 651--664.
\bibitem{Xi} J. Xiao, The $Q_p$ Carleson measure problem,  Adv. Math. {\bf 217} (2008), 2075--2088.
\bibitem{Z} K. Zhu,  Operator Theory in Function Spaces. Second edition. Mathematical Surveys and Monographs, 138. American Mathematical Society, Providence, RI, 2007.
\end{thebibliography}
\end{document}